\documentclass[reqno,11pt,twoside,english]{amsart}
\usepackage[T1]{fontenc}
\usepackage{amsthm}
\usepackage{amssymb}
\usepackage{amsmath,amsfonts,epsfig}
\usepackage{mathrsfs}
\usepackage{float}
\makeatletter
\newcommand\fermatspiral{} 
\def\fermatspiral[#1](#2)(#3:#4)[#5]{%
	\pgfmathsetmacro{\domain}{#4+#5*360}
	\draw [#1,
	shift={(#2)},
	domain=#3*pi/180:\domain*pi/180,
	variable=\t,
	smooth,
	samples=int(\domain/5)] plot ({\t r} : {sqrt(\t)})
}
\numberwithin{equation}{section}
\usepackage{color}
\usepackage[final,linkcolor = blue,citecolor = blue,colorlinks=true]{hyperref}
\usepackage{type1cm,ae}
\usepackage{graphicx}
\usepackage{xspace}
\usepackage{setspace}
\usepackage[english]{babel}
\usepackage{tikz}
\usetikzlibrary{arrows,patterns}
\newcommand{\rr}{{\mathbb R}}
\newcommand{\R}{{\mathbb R}}

\newcommand\loc{{\mathop\mathrm{\,loc\,}}}

\newcommand\diam{{\mathop\mathrm{\,diam\,}}}

\newcommand*{\bR}{\ensuremath{\mathbb{R}}}

\numberwithin{equation}{section}
\newtheorem{theorem}{Theorem}[section]

\newtheorem{lemma}[theorem]{Lemma}

\newtheorem{proposition}[theorem]{Proposition}

\theoremstyle{definition}
\begingroup
\newtheorem{definition}[theorem]{Definition}
\newtheorem{remark}[theorem]{Remark}
\newtheorem{example}[theorem]{Example}
\endgroup

\def\XXint#1#2#3{{\setbox0=\hbox{$#1{#2#3}{\int}$}
		\vcenter{\hbox{$#2#3$}}\kern-.5\wd0}}

\makeatother
\usepackage{babel}
\begin{document}
	\title[Boundary extensions of mappings between metric spaces]{Boundary extensions for mappings between metric spaces}
	\author[Y.-L Tian and Y. Xuan]{Yao-Lan Tian and Yi Xuan}
	
	\address[Yao-Lan Tian]{Center for Optics Research and Engineering, Shandong University, Qingdao 266237, P. R. China}
	\email{tianylsdu@sdu.edu.cn}

	\address[Yi Xuan]{Research Center for Mathematics, Shandong University 266237,  Qingdao, P. R. China and  Frontiers Science Center for Nonlinear Expectations, Ministry of Education, P. R. China}
	\email{yixuan@sdu.edu.cn}
%
%
%
	
	\thanks{The authors are supported by the Young Scientist Program of the Ministry of Science and Technology of China (No.~2021YFA1002200), the National Natural Science Foundation of China (No.~12101362) and the Natural Science Foundation of Shandong Province (No.~ZR2021QA003).}

\begin{abstract}
In this paper, we consider boundary extensions of two classes of mappings between metric measure spaces. These two mapping classes extend in particular the well-studied geometric mappings such as quasiregular mappings with integrable Jacobian determinant and mappings of exponentially integrable distortion with integrable Jacobian determinant. Our main results extend the corresponding results of \"Akkinen and Guo [Ann. Mat. Pure. Appl. 2017] to the setting of metric measure spaces.
\end{abstract}

\maketitle

{\small
\keywords {\noindent {\bf Keywords:} Uniform domains, $\varphi$-length John domains, Dyadic-Whitney decomposition, limits along John curves, quasiregular mappings.}
\smallskip
\newline
\subjclass{\noindent {\bf 2010 Mathematics Subject Classification: 49N60; 58E20}   }
}
\bigskip

\arraycolsep=1pt

\section{Introduction}

Let $\Omega\subset\bR^n$($n\geq 2$) be a domain.
A mapping $f\colon\Omega\to\bR^n$ is called a mapping of finite distortion if the following three conditions are satisfied:
\begin{enumerate}
	\item $f\in W^{1,1}_{\loc}(\Omega,\bR^n)$,
	\item $J_f=\textup{det}(Df)\in L^1_\text{loc}(\Omega)$,
	\item there exists a measurable function $K_f\colon\Omega\to[1,\infty]$ so that for almost every $x\in\Omega$, \[K_f(x)<\infty,\] and
\[|Df(x)|^n\leq K_f(x)J_f(x),\]
\end{enumerate}
where $|\cdot|$ is the operator norm. If $K_f\leq K<\infty$ almost everywhere, we say that $f$ is $K$-quasiregular. If $n=2$ and $K=1$, we recover the concept of complex analytic functions. A mapping $f$ of finite distortion is called a mapping of exponentially integrable distortion if $\exp(\lambda K_f)\in L^1(\Omega)$ for some $\lambda>0$.

The theory of quasiregular mappings, initiated by the work of Reshetnyak, Martio, Rickman and V\"ais\"al\"a, shows that they form, from the geometric
function theoretic point of view, the correct generalization of the class of analytic functions to higher dimensions. In particular, Reshetnyak proved that non-constant quasiregular mappings are continuous, discrete and open, and that they preserve sets of measure zero; see \cite{r89}, \cite{r93} for the theory of quasiregular mappings. This theory was extended to the setting of mappings of exponentially integrable distortion in the fundamental work of Iwaniec, Koskela and Onninen \cite{iko01} and attracted great attension; see \cite{hk14,im01} for a comprehensive introduction of this theory.

In this article, we are mainly interested in the boundary behavior of this class of mappings in the setting of metric measure spaces. Recall that a classical theorem of Fatou states bounded analytic functions defined on
the unit disc have radial limits at almost every boundary point. It is a well-known open problem whether this classical theorem remains valid  for quasiregular mappings in higher dimensions $n\geq 3$. It is indeed even unknown whether bounded quasiregular maps have radial limits at any point on the boundary of the unit ball. A substantial progress on this problem was given by Rajala \cite{r08}, who proved that quasiregular local homeomorphisms have radial limits at infinitely many boundary points of the unit ball. This result was extended recently by Guo and Xuan \cite{gx23} to quasiregular local homeomorphisms defined on more general domains. Adding certain growth assumptions on the multiplicity functions, Martio and Rickman \cite{mr72} proved the existence of radial limits of quasiregular mappings at almost every boundary point. This result was further extended by \"Akkinen \cite{a14} to mappings of finite distortion defined on the unit ball, by \"Akkinen and Guo \cite{ag17} to mappings of finite distortion defined on general John domains and by Cristea \cite{c23} for quasiregular mappings defined on Riemannian manifolds. For homeomorphisms of finite distortion, an essentially sharp result along this direction was obtained earlier by Koskela and Nieminen \cite{kn08}; see also \cite{ms75} for examples of quasiregular mappings without radial limits at any boundary point.

We now turn to the setting of metric measure spaces. The theory of quasiconformal mappings in the setting of metric measure spaces was initiated by Heinonen and Koskela in \cite{hk98} and was further developed in \cite{bkr07,w12,w14} (see also the references therein for more related works). The theory of quasiregular mappings and mappings of finite distortion in singular metric spaces was initiated by Heinonen and Holopainen \cite{hh97} (for Carnot groups) and was later extended to more general setting in \cite{hr02,or09,c06,g15b,gw16}.

Let $X_{1}=(X_{1},d_{1},\nu_{1})$, $X_{2}=(X_{2},d_{2},\nu_{2})$ be two path-connected complete metric measure spaces and $\Omega$ be a bounded domain of $X_{1}$. Through this paper, we shall assume that $\Omega\subset X_{1}$ is proper and Ahlfors $q$-regular (see Section \ref{sec:pre} below for precise definitions). For the statement of our main result, we introduce the following two classes of mappings.

\begin{definition}\label{def 1.1}
Let $f:\Omega\rightarrow X_{2}$ be a mapping and $\sigma\geq 1$ be a constant.
\begin{enumerate}
\item We say that $f$ belongs to the class $\mathcal{A}_{1}(\sigma)$, if there exist a function $\alpha\in L^{1}(\Omega)$ and a constant $C\geq 1$, such that
$$\diam f(B)\leq C\bigg(\int_{\sigma B}\alpha(y)d\nu_{1}(y)\bigg)^{\frac{1}{q}}$$
for every $B=B(x,r)$ for which $\sigma B\subset\subset\Omega$.
		
\item We say that $f$ belongs to the class $\mathcal{A}_{2}(\sigma)$,  if there exist a function $\alpha\in L^{1}(\Omega)$ and a constant $C\geq 1$ such that
    $$\diam f(B)\leq C\bigg(\int_{\sigma B}\alpha(y)d\nu_{1}(y)\bigg)^{\frac{1}{q}}\log\bigg(\frac{1}{\diam B}\bigg)^{\frac{1}{q}}$$
    for every $B=B(x,r)$ for which $\sigma B\subset\subset\Omega$.
	\end{enumerate}
\end{definition}
In the setting of Euclidean spaces, the class $\mathcal{A}_1$ contains all quasiregular mappings with integrable Jacobian determinant between Euclidean domains, while $\mathcal{A}_2$ contains mappings of exponentially integrable distortion with integrable Jacobian determinant. In both cases, $\alpha$ can be taken as the Jacobian determinant $J_f$ of $f$; see for instance \cite{ag17,kn08}. We shall show in Example \ref{exam:two classes} below that this extends to the setting of sufficiently nice metric measure spaces as well.


Next, we introduce the class of $\varphi$-length John domains that were initially introduced in \cite{gk14}.
\begin{definition}\label{length john curve}
	Let $\varphi$ be a continuous, increasing function with $\varphi(0)=0$ and $\varphi(t) \geq t$ for all $t>0$. A bounded domain $\Omega\subset X_{1}$ is called a $\varphi$-length John domain with center $x_{0}\in\Omega$ if there exists a constant $c\geq1$ such that for each point $x\in\Omega$, there exists a curve $\gamma:[0,1]\rightarrow\Omega$ satisfying $\gamma(0)=x,\gamma(1)=x_{0}$ and for all $t\in[0,1]$, \begin{equation}\label{ccc}
	\varphi(cd_{1}(\gamma(t), \partial \Omega)) \geq \ell(\gamma([0,t])).
\end{equation}
\end{definition}
The curve $\gamma$ appearing in Definition \ref{length john curve} is called the $\varphi$-length John curve connecting $x$ to the center $x_0$ and it needs not to be unique. Given $\xi\in\partial\Omega$, the set $I^{(\varphi,c)}(\xi,x_{0})$ consists of all $\varphi$-length John curves in $\Omega$ connecting $\xi$ to $x_{0}$ with constant $c$ in the above inequality (see Lemma \ref{connect to boundary} below for the existence).

When $\varphi(t)=t$,  it reduces to the well-known class of $c$-John domains with center $x_0$ (see for instance \cite{hak00} for more on John domains in the setting of metric spaces), where $c$ comes from \eqref{ccc}. When $\varphi(t)=t$, we simply write $I^{c}(\xi,x_{0})$, instead of $I^{(\varphi,c)}(\xi,x_0)$, to denote the class of all $c$-John curves connecting $\xi$ to $x_{0}$, where $c$ is also from \eqref{ccc}.

The purpose of this paper is to prove that the functions $f\in\mathcal{A}_{i}(\sigma),$ $i=1,2$, for a sufficiently large portion of the boundary, have limits along $\varphi$-length John curves. Recall that the generalized Hausdorff $h$-measure of a set $E$ in a metric space is defined as
\[H^h(E)=\lim_{r\rightarrow 0}\inf \left\{\sum h(\diam E_i):E\subset \bigcup_i E_i, \diam E_i\leq r, i\in \mathbb{N}\right\},\]
where $E_i$ are subsets of $E$ and the gauge function $h:[0,\infty)\rightarrow [0,\infty)$ is continuous and increasing such that $h(0)=0$. If $h(t)=t^\alpha$ for some $\alpha>0$, then this concept returns to the usual concept of $\alpha$-dimensional Hausdorff measure. Notice that we say a positive function $h(t)$ with $t>0$ is doubling if there exists a constant $C$ such that
$$h(2t)\leq Ch(t) \quad \text{for all } t>0.$$
\begin{theorem}\label{thm:limit exist}
Let $\Omega\subset X_{1}$ be a $\varphi$-length John domain with center $x_{0}$. Suppose $\Omega$ is Ahlfors $q$-regular. Given a map $f:\Omega\rightarrow X_{2}$, let $E_{f}$ be the set of points $\omega\in\partial\Omega$ for which there exists a curve $\gamma\in I^{(\varphi,c)}(\omega,x_{0})$ so that $f$ does not have a limit along $\gamma$. Then there is some $\sigma\geq 1$ depending only on the data associated to $\Omega$ with the following properties:
	\begin{enumerate}
		\item Let h be a doubling gauge function such that $$\int_{0}^{1} {\left[\frac{\varphi(t) }{t}\right]}^qh(\varphi(t))^ {\frac{1}{q-1}} \frac{dt}{t}<\infty.$$
		If $f\in\mathcal{A}_{1}(\sigma)$, then $\mathcal{H}^{h}(E_{f})=0$.
		\item Let h be a doubling gauge function such that \begin{equation}\label{210}\int_{0}^{1} {\left[\frac{\varphi(t) }{t}\right]}^qh(\varphi(t))^ {\frac{1}{q-1}}\left[\log \frac{1}{t}\right]^{\frac{1}{q-1}} \frac{dt}{t}<\infty.\end{equation}
		If $f\in\mathcal{A}_{2}(\sigma)$, then $\mathcal{H}^{h}(E_{f})=0$.
	\end{enumerate}
\end{theorem}

Theorem \ref{thm:limit exist} can be regarded as a natural extension of \cite[Theorem 3.5]{ag17} and \cite[Corollary 1.3]{kn08} from Euclidean domains to domains in general metric measure spaces, and from linear John domains to general nonlinear $\varphi$-length John domains.

In case $\Omega$ is a uniform domain, we are able to prove the uniqueness of limits. Recall that a bounded domain $\Omega\subset X_{1}$ is called a uniform domain if there exists a constant $c>0$ such that each pair of points $x_{1},x_{2}\in\Omega$ can be joined by a curve $\gamma$ in $\Omega$ with the following two properties:
\begin{align}
\ell(\gamma) &\leq cd_{1}(x_{1},x_{2}), \\
	\min_{i=1,2}\ell(\gamma(x_{i},x)) &\leq cd_{1}(x,\partial\Omega)\text{ for all } x\in \gamma.
\end{align} It is not hard to prove that uniform domains are John domains. As in \cite{ag17}, we say that $\Omega$ is a $c$-uniform domain with center $x_{0}$, if $\Omega$ is a uniform domain and, as a John domain, it is $c$-John with center $x_{0}$.

Our second main result reads as follows, which generalizes \cite[Theorem 3.10]{ag17}.
\begin{theorem}\label{thm:limit is unique}
Let $X_1$ be a metric measure space and $\Omega \subset X_{1}$ be a c-uniform domain with center $x_{0}$. Assume that $\Omega$ is Ahlfors $q$-regular and $h$ is a doubling gauge function satisfying \[\int_{0}^{1} \frac{\left[h(t)\log \frac{1}{t}\right]^{\frac{1}{q-1}}}{t} d t<\infty.\] There is some $\sigma\geq 1$ depending only on the data of $\Omega$ such that if $f\in\mathcal{A}_{2}(\sigma)$, then $f$ has unique limits along all curves $\gamma \in I^{c}\left(w, x_{0}\right)$ for $\mathcal{H}^{h}$-almost every $\omega \in \partial \Omega$, that is, for any $\gamma, \eta \in I^{c}\left(\omega, x_{0}\right)$ satisfying
 	$$
 	\lim _{t \rightarrow 0^{+}} f(\gamma(t))=a \text { and } \lim _{t \rightarrow 0^{+}} f(\eta(t))=b,
 	$$
 we have $a=b$.
 \end{theorem}

The idea for proving both Theorem~\ref{thm:limit exist} and Theorem~\ref{thm:limit is unique} is very similar to that used in \cite{ag17} and thus part of the arguments are similar to the one used there. On the other hand, some of the arguments from \cite{ag17} depends on the Euclidean geometry, and thus do not apply to the general setting of metric spaces. To overcome this technical difficulty, we first develop a suitable Dyadic-Whitney decomposition for a proper domain in a general doubling metric space, which relies crucially on the result of Hyt\"onen and Kairema \cite{hk12}. Then we provide a new argument, which does not depend on the geometry of Whitney cubes in Euclidean spaces, to prove the uniqueness of limits along John curves, namely, Theorem \ref{thm:limit is unique}.

This paper is organized as follows. In Section \ref{sec:pre}, some preliminaries and auxiliary results are presented. In Section \ref{Existence of limits along curves}, we will show that the mapping $f$ can be extended to the boundary of a $\varphi$-length John domain along $\varphi$-length John curves with a small exceptional set with respect to almost sharp Hausdorff gauges, and in Section \ref{Uniqueness of limits along John curves}, we prove our uniqueness theorem.

\section{Preliminaries and auxiliary results}\label{sec:pre}
A curve $\gamma$ in $\Omega$ is a continuous mapping $\gamma:[0,1] \rightarrow \Omega$. A curve $\gamma:[0,1]\rightarrow\Omega$ is said to connect points $x, y \in \Omega$, if $\gamma(0)=y$ and $\gamma(1)=x$; similarly a curve $\gamma:(0,1]\rightarrow\Omega$ is said to connect points $x \in \Omega, y \in \partial \Omega$ if $\gamma(1)=x, \gamma((0,1]) \subset \Omega$ and
$$
\lim _{t \rightarrow 0^{+}} \gamma(t)=y.
$$
We use the notation $\ell(\gamma)$ to denote the Euclidean length of a curve $\gamma$.

Throughout this paper, $C_i(\cdot)$ denotes a constant, where ``$\cdot$" contains all parameters on which the constant depends and $i\geq 1$ is an integer. Moreover, $A\lesssim B$ means $A\leqslant CB$, where $C>0$ is a constant. Furthermore, we say $A\approx B$, if $A\lesssim B\lesssim A$.
\begin{definition}[Ahlfors $q$-regularity]\label{def q-regular}
Let $X_{1}=(X_{1},d_{1},\nu_{1})$ be a metric measure space. An open set $\Omega\subset X_{1}$ is said to be Ahlfors $q$-regular ($q>1$) if there exists a contant $C_{q}\geq 1$ such that for any $x \in \Omega$ and for any $0<r<\diam \Omega$,  $$\frac{1}{C_{q}}r^q\leq \nu_{1}\bigg(B(x,r)\cap\Omega\bigg)\leq C_{q} r^q.$$
\end{definition}

A weaker condition than Ahlfors regularity is the following doubling property.
\begin{definition}[Doubling metric space]
	A metric space $(X,d)$ is said to be a doubling metric space if there exists a positive integer $M\in\mathbb{N}$ such that for every $x\in X$ and for every $r>0$, the
	ball $B(x,r):=\{y\in X, d(y,x)<r\}$ can be covered by at most $M$ balls $B(x_{i},\frac{r}{2})$, $1\leq i\leq M$
\end{definition}

The following theorem, proved by Hyt\"onen and Kairema \cite{hk12}, plays an important role in establishing the Dyadic-Whitney decomposition for proper domains in a doubling metric space.
\begin{proposition}[\cite{hk12}]\label{dyadic decomposition}
Let $X=(X,\rho)$ be a doubling metric space and assume that there are constants $0<c_{0} \leq C_{0}<\infty$ and $\delta \in(0,1)$ such that $12C_{0} \delta \leq c_{0}$. Given a set of points $\left\{z_{\alpha}^{k}\right\}_{\alpha}, \alpha \in \mathscr{A}_{k}$, $k \in \mathbb{Z}$, satisfying that
	$$
	\rho\left(z_{\alpha}^{k}, z_{\beta}^{k}\right) \geq c_{0} \delta^{k} \quad(\alpha \neq \beta), \quad \min _{\alpha} \rho\left(x, z_{\alpha}^{k}\right)<C_{0} \delta^{k} \quad \text {for all x} \in X.
	$$
Then we can construct a family of sets $Q_{\alpha}^{k}$, called half-open dyadic cubes, satisfying that:
\begin{itemize}
\item[(1)] $X=\bigcup_{\alpha} Q_{\alpha}^{k} \quad \text{ (disjoint union) } \quad \text{for all } k \in \mathbb{Z}$;
\item[(2)] if $l \geq k$, then either $Q_{\beta}^{l} \subseteq Q_{\alpha}^{k}$ or $Q_{\alpha}^{k} \cap Q_{\beta}^{l}=\emptyset$;
\item[(3)] $B\left(z_{\alpha}^{k}, c_{1} \delta^{k}\right) \subseteq Q_{\alpha}^{k} \subseteq B\left(z_{\alpha}^{k}, C_{1} \delta^{k}\right)$, where $c_{1}:=3^{-1} c_{0}$ and $C_{1}:=2C_{0}.$
\end{itemize}
The half-open cubes $Q_{\alpha}^{k}$ depend on $z_{\beta}^{l}$ for $l \geq \min \left(k, k_{0}\right)$, where $k_{0} \in \mathbb{Z}$ is a preassigned number.
\end{proposition}

Let $X$ be a doubling metric space. Using the dyadic cubes from Proposition~\ref{dyadic decomposition}, we can select a suitable subfamily of cubes to form a Whitney decomposition of a proper domain $\Omega\subset X$. Recall that the Whitney decomposition states that any proper domain $\Omega\subset \rr^n$, i.e., $\R^n\backslash \Omega\neq \emptyset$,
can be written as the union of cubes such that the diameter of each cube is comparable to its distance to the boundary. As a substitute for the Whitney decomposition in the Euclidian setting, we introduce the Dyadic-Whitney decomposition in a doubling metric spaces.
\begin{definition}\label{dyadic-whitney decomposition}(Dyadic-Whitney decomposition)
Let $X_{1}=(X_{1}, d_{1})$ be a doubling metric space. For a proper domain $\Omega \subseteq X_{1}$, a Dyadic-Whitney decomposition of $\Omega$ with data $(\delta,c_1,C_1,a)$, where $\delta\in (0, 1)$, $0<c_1<C_1$, and $a \geq 4$, is a collection $\mathscr{W}_{\Omega}$ of subsets of $X_{1}$ satisfying:
	\begin{itemize}
		\item[(1)] $\bigcup\limits_{Q\in\mathscr{W}_{\Omega}}Q$ = $\Omega$;
		\item[(2)] 	$Q\bigcap Q'=\emptyset$ for all $Q,Q'\in \mathscr{W}_{\Omega}$, where $Q\neq Q'$;
		\item[(3)] For any $Q\in \mathscr{W}_{\Omega}$, there exists $x=x_Q\in \Omega$ and $k\in \mathbb{Z}$ such that
		$$B(x,c_1\delta ^k)\subset Q\subset B(x,C_1\delta ^k)$$
		and
		$$(a-2)C_1\delta ^k\leq d_{1}(Q,X\backslash\Omega)\leq aC_1\delta ^{k-1}.$$
	\end{itemize}
\end{definition}	

The sets $Q\in \mathscr{W}_{\Omega}$ are referred to as Dyadic-Whitney cubes. As each Dyadic-Whitney cube $Q$ can be assigned with  a number $k\in\mathbb{Z}$ satisfying the above properties, we can express $\mathscr{W}_{\Omega}$ as the union of the disjoint subsets ${\mathscr{W}_{\Omega}^k}$, $k\in \mathbb Z$, where $\mathscr{W}_{\Omega}^k$ is a collection of dyadic cubes with the above properties for fixed $k$ and that the union of these dyadic cubes is a connected set. Moreover, set $B^Q=B(x,C_1\delta^k)$ and $B_Q=B(x,c_1\delta^k)$.

We first establish the folowing existence result.

\begin{lemma}\label{existence of decomposition}
	Let $X=(X,\rho)$ be a doubling metric space and $\Omega\subset X$ a proper domain. For all $\delta, c_{1}, C_{1}, a>0$ satisfying $a\geq 4$, $6c_1\leq C_1$,$2C_1\delta\leq c_1$ and $\delta\in(0,1)$, there exists a Dyadic-Whitney decomposition of $\Omega$ with data $\left(\delta, c_{0}, C_{1}, a\right)$.
\end{lemma}
\begin{proof}
Assume that $C_0=C_1/2$ and $c_0=3c_1$. Observe that $C_0, c_0, C_1, c_1$ and $\delta$ satisfy the requirement of Proposition~\ref{dyadic decomposition}. It is not hard to see that there exists a set of points $\left\{z_{\alpha}^{k}\right\}_{\alpha}, \alpha \in \mathscr{A}_{k}$, for every $k \in \mathbb{Z}$, with the properties that
	$\rho\left(z_{\alpha}^{k}, z_{\beta}^{k}\right) \geq c_{0} \delta^{k}$ when $\alpha \neq \beta$, $\min _{\alpha} \rho\left(x, z_{\alpha}^{k}\right)<C_{0} \delta^{k} \text { for all } x\in X$. Thus, we can construct a family of sets $Q_{\alpha}^{k} \subset \Omega$ satisfying the properties in Proposition \ref{dyadic decomposition}. Then, we divide $\Omega$ as
$$\Omega_{k}=\left\{x\in\Omega: a C_{1} \delta^{k}<d(x, X \backslash \Omega) \leq a C_{1} \delta^{k-1}\right\},
$$ such that $\Omega=\bigcup\limits_{k=-\infty}^{\infty} \Omega_{k}$.
Set
$$\mathscr{W}_{\Omega}^{0}=\bigcup_{k}\left\{Q_{\alpha}^{k}: \alpha \in \mathscr{A}_{k}, Q_{\alpha}^{k} \cap \Omega_{k} \neq \emptyset\right\}.
$$
Now fix $Q^k_{\alpha}\in\mathscr{W}_{\Omega}^{0}$. Then, $d\left(Q_{\alpha}^{k}, X \backslash \Omega\right) \leq a C_{1} \delta^{k-1}$, and it follows from the triangle inequality that \begin{equation}\label{120}
d\left(Q_{\alpha}^{k}, X \backslash \Omega\right) \geq a C_{1} \delta^{k}-\diam Q_{\alpha}^{k} \geq(a-2) C_{1} \delta^{k}.
\end{equation}
Also, note that
\begin{equation}\label{110}
	\diam Q_{\alpha}^{k}\approx \delta^{k}\approx d(Q_{\alpha}^{k},X\backslash\Omega),
\end{equation}
where the constants are independent of $\alpha$ and $k$. Then, by the above inequality, we easily deduce that these cubes are disjoint from $X\backslash \Omega$ and clearly cover $\Omega$, that is, \[\bigcup_{Q\in\mathscr{W}_{\Omega}^{0}}Q=\Omega.\]	
Notice that $\mathscr{W}_{\Omega}^{0}$ satisfies all the requirement except that these cubes are not disjoint.

Next, we refine the choice of cubes. Fix any cube $Q\in\mathscr{W}_{\Omega}^{0}$ and consider all the cubes  $Q^{\prime}$ which contain $Q$. Then, by \eqref{110}, we have \[\diam Q^{\prime}\approx d(x_0,X\backslash\Omega)\approx\diam Q,\] where $x_0$ is a point in $Q$. Thus, the number of such cubes $Q^{\prime}$ is finite. Since by Proposition~\ref{dyadic decomposition} any two cubes in $\mathscr{W}_{\Omega}^{0}$ are either disjoint or one contains the other, we know that there exists a unique maximal cube $Q^{\prime}$ containing $Q$. Denote the set of all the maximum cubes $Q^{\prime}$ as $\mathscr{W}_{\Omega}$. Then, obviously $\mathscr{W}_{\Omega}$ satisfies all the properties of the dyadic decomposition.
\end{proof}

\begin{remark}\label{rmk:on Whitney decomposition}
According to Lemma \ref{existence of decomposition}, if $\delta, c_0,C_1, a$ satisfy $0<\delta<1$, $a\geq 4$, $6c_1\leq C_1$, $2C_1\delta\leq c_1$, then the dyadic-Whitney decomposition of $\Omega$ exists. Moreover, \eqref{120} implies that
\[d(B^Q,X\backslash \Omega)=d(B(x,C_1\delta^k),X\backslash \Omega)\geq C_1\delta^k.\]
In particular, we have
\begin{equation}\label{150}
\frac{3}{2}B^Q\subset\subset\Omega.
\end{equation}
This shows that the Dyadic-Whitney decomposition is different from the classical Whitney decomposition, even for domains in the Euclidean space $\R^n$.
\end{remark}

For the next lemma, let $\lambda B^{Q}$ be the ball which has the same center as $B^{Q}$ but the radius is expanded by the factor $\lambda$.

\begin{lemma}\label{lambda}
	 Under the condition of Lemma $\ref{existence of decomposition},$ there exists a constant $\lambda_{0}>1$ such that for any $x\in\Omega$ and $1\leq \lambda\leq\lambda_{0}$, there are at most $N=N\left(\delta, c_{1}, C_{1}, a\right)$ balls $\lambda B^{Q}$, $Q\in\mathscr{W}_{\Omega}$, containing $x$.
\end{lemma}
\begin{proof}
Set $\lambda_0=\frac{3}{2}$. Fix $\lambda$ satisfying $1\leq \lambda\leq \lambda_0$ and $x\in\Omega$. Consider a dyadic cube $Q$ such that $\lambda B^Q$ contains $x$. Set the center of $B^Q$ to be $x_Q$. Then, \begin{equation}\label{130}
d(x,X\backslash \Omega)\geq d(\lambda B^Q,X\backslash\Omega)\geq d(x_Q,X\backslash\Omega)-\lambda C_1\delta^k\geq 2C_1\delta^k-\frac{3}{2}C_1\delta^k=\frac{1}{2}C_1\delta^k,\end{equation}
where $k$ is the index such that the dyadic cube $Q$ satisfies (3) of Proposition~\ref{dyadic decomposition}.
Thus, we find that
\[\diam(B^Q)\leq 2 C_1\delta^k\leq 4d(x,X\backslash \Omega).\]
Therefore,
\[B^Q\subset B(x,4d(x,X\backslash \Omega)).\]
Moreover,
\[d(x,X\backslash \Omega)\leq d(Q,X\backslash \Omega)+\diam(\lambda B^Q)\leq aC_1\delta^{k-1}+2\lambda C_1\delta^k=:C\delta^k.\]
Then, by the above inequality and \eqref{130}, we infer that
\[
\diam(B_Q)\approx d(x,X\backslash \Omega)
\]
and that
\[
B_Q\subset B(x,4d(x,X\backslash \Omega))
\] for all such $Q$. Thus, by the doubling property, there are at most $N=N\left(\delta, c_{1}, C_{1}, a\right)$ dyadic cubes $Q$ such that $\lambda B^{Q}$ contains $x$.

\end{proof}

In the next example, we point out that when the spaces are sufficiently nice, the class $\mathcal{A}_1$ contains quasiregular maps with integrable Jacobian determinant and $\mathcal{A}_2$ contains mappings of exponentially integrable distortion with integrable Jacobian determinant.

\begin{example}\label{exam:two classes}
Suppose $X_1=(X_1,d_1,\nu_1)$ is a proper metric measure space and $\Omega\subset X_1$ is a bounded Ahlfors $q$-regular domain supporting a $(1,p)$-Poincar\'e inequality with $p>q-1$ in the sense of Heinonen and Koskela \cite{hk98}\footnote{We recommend the readers to the monograph \cite{hkst15} for more information}. Assume $X_2=(X_2,d,\nu_2)$ is a complete metric measure space, satisfying condition (P): for any two points $y_1, y_2\in X_2$ and any ball $B(y_2,\delta)\subset X_2$, there exists a point $y_2^{\prime}\in B(y_2,\delta)$ such that \[d(y_2^{\prime},y_1)> d(y_2,y_1).\]
Let $f\colon \Omega\to X_2$ be a non-constant continuous $N^{1,p}(\Omega,X_2)$-map. Then there exists a constant $\sigma\geq 1$ such that
\begin{itemize}
	\item $f\in \mathcal{A}_1(\sigma)$ whenever it is $K$-quasiregular with integrable Jacobian determinant;
	\item $f\in \mathcal{A}_2(\sigma)$ whenever it is of $\lambda$-exponentially integrable distortion\footnote{See \cite{g15b} for precise definition} and with integrable Jacobian determinant.
\end{itemize}
We briefly indicate the proof of the first assertion, since the proof of the second one is similar (both are indeed similar to \cite[Proof of Theorem 1.1]{kn08}). Note that the only difference with the previous theorem is the usage of the condition (P) for open and continuous mappings.

Fix a ball $B=B(x,r)$ with $5\sigma B\subset \subset\Omega$, where $\sigma$ is the expanding factor of the balls from the Poincar\'e inequality of the space $X_1$. By the abstract Sobolev embedding on spheres from \cite[Theorem 7.1]{hak00}, we know
\begin{equation}\label{eq:Sob embedding sphere}
	\diam f(\partial B(x,r))\leq r^{\frac{p-q}{p}}\Big(\int_{5\sigma B}g_f(y)^pd\nu(y) \Big)^{\frac{1}{p}}
\end{equation}
for some $r\in (r_0/2,r_0)$, where $g_f$ is the minimal $p$-weak upper gradient of $f$; see \cite{hkst15}.
Since $f$ is $K$-quasiregular, $g_f^q\leq KJ_f$ and so by H\"older's inequality, we know
$$\Big(\int_{5\sigma B}g_f(y)^pd\nu(y) \Big)^{\frac{1}{p}}\leq K^{\frac{1}{q}}r^{\frac{q-p}{p}}\Big(\int_{5\sigma  B}J_f(x)d\nu_1(x) \Big)^{1/q}.$$
Next, we claim that
$$\diam f(B)\leq \diam f(\partial B)$$
for every ball $B$ with $B\subset \subset\Omega$. We prove by contradiction a little stronger statement: \[\diam(f(\overline{B}))\leq\diam(f(\partial B)).\] Assume the statement is false. Then, \begin{equation}\label{140}
	\diam(f(\overline{B}))>\diam(f(\partial B)).
\end{equation}
Since $\overline{B}$ is compact and $f$ is continuous, there are two points $x$ and $y$ in the closed ball such that $d(f(x),f(y))=\diam (f(\overline{B}))$. Note that $x,y$ can not belong to $\partial B$ simultaneously, since if it happens, then \eqref{140} is false. Without loss of generality, we may assume that $x\in B$. Then, as $f(B)$ is open, there is a ball $B^{\prime}$ in $X_2$ such that $f(x)\in B^{\prime}\subset f(B)$. Then, Using the condition (P), we find a point $f(z)\in B^{\prime}$ for some $z\in B$ such that
$$d(f(z),f(y))>d(f(x),f(y))=\diam(f(\overline{B})),$$
which contradicts with \eqref{140}. Then, using the above claim for each $B=B(x,r)$ with $5\sigma B\subset\subset \Omega$,
$$\diam f(B)\leq \diam f(\partial B)\leq C\Big(\int_{5\sigma B}J_f(x)d\nu_1(x)\Big)^{1/q}.$$
\end{example}

\section{Proof of the main results}
\subsection{Existence of limits along curves}\label{Existence of limits along curves}
There might be points on the boundary $\partial\Omega$ that are not accessible by a rectifiable curve inside $\Omega$ if we just think about general bounded domains; see for instance \cite[Figure 1]{ag17}.
To avoid these situations, the authors considered in~\cite{ag17} the class of John domains. Here, we consider the more general class of $\varphi$-length John domains.
\begin{lemma}\label{connect to boundary}
	Let $\Omega\subset X_{1}$ be a $\varphi$-length John domain with center $x_{0}$. Then each boundary point can be connected to $x_{0}$ by a $\varphi$-length John curve.
\end{lemma}
\begin{proof}
	For any $\omega \in \partial \Omega$, we can find a sequence $\omega_{i} \in \Omega$ such that $\lim\limits_{i \rightarrow+\infty} \omega_{i}=\omega .$ Since $\Omega$ is a $\varphi$-length John domain, every $\omega_{i} \in \Omega$ can be connected to $x_{0}$ by some curve $\gamma_{i}:[0,1] \rightarrow \Omega$ such that for all $t \in[0,1]$,
	$$
	\ell\left(\left.\gamma_{i}\right|_{[0, t]}\right) \leq \varphi\left(c d_{1}\left(\gamma_{i}(t), \partial \Omega\right)\right).
	$$
	Thus
	$$
	L:=\operatorname{sup}_{i}l_{i}=\operatorname{sup}_i\ell\left(\gamma_{i}\right) \leq \varphi(\operatorname{cdiam} \Omega)<\infty.
	$$
	If we reparameterize $\gamma_{i}$ by arc-length, then $\gamma_{i}:\left[0, l_{i}\right] \rightarrow \Omega$ is 1-Lipschitz. According to the Arzela-Ascoli theorem, there exists $\gamma:[0, L] \rightarrow \overline{\Omega}$ such that $\gamma_{i}$ converge to $\gamma$ uniformly. Since the length functional is lower semicontinuous (with respect to uniform convergence) and $\varphi, d_{1}$ are continuous, we have for all $t\in [0,L],$
	$$
	\ell\left(\left.\gamma\right|_{[0, t]}\right) \leq \varphi\left(c d_{1}(\gamma(t), \partial \Omega)\right).
	$$
	Therefore, $\gamma$ is a desired $\varphi$-length John curve connecting $\omega$ to $x_{0}$.
\end{proof}
We define $P^{(\varphi,c)}(\xi)=\{Q\in \mathscr{W}_{\Omega}:Q\cap\gamma\neq\emptyset  \text{ for some } \gamma\in I^{(\varphi,c)}(\xi,x_{0})\}.$ For $Q\in \mathscr{W}_{\Omega}$ and $ E\subset\partial\Omega$, the shadow of $Q$ on $E$ is defined by $$S^{(\varphi,c)}_{E}(Q)=\{\xi\in E:Q\in P^{(\varphi,c)}(\xi)\}.$$ If $E=\partial\Omega$, we write $S^{(\varphi,c)}(Q)$ instead of $S^{(\varphi,c)}_{\partial\Omega}(Q)$.

Notice that for a point $y\in\partial\Omega$ there might be infinitely many $\varphi$-length John curves connecting $x_{0}$ and $y$. For the rest of this section, the standing assumptions are: $\Omega$ is a $\varphi$-length John domain with center $x_{0}$ and $\mathscr{W}_{\Omega}$ is a dyadic-Whitney decomposition of $\Omega$.

We need the following two basic estimates for the shadow $S^{(\varphi,c)}$. In the setting of Euclidean spaces, these results were proved in \cite{g15}.

\begin{lemma}\label{geshu}
	Let $Q\in \mathscr{W}_{\Omega}$. Then $S^{(\varphi,c)}(Q)$ is closed and there exists a constant $C=C(c,c_{0},C_{1},\delta,a)>0$ such that
	$$\diam S^{(\varphi,c)}(Q)\leq 3\varphi(C\diam Q).$$
	Furthermore, for any $k\in \mathbb{Z}$,
	$$\#\{Q\in \mathscr{W}_{\Omega}^k:\xi \in S^{(\varphi,c)}(Q)\}\leq C_{q}^{2}\Big(\frac{\varphi(2CC_1\delta^k)}{c_1\delta^k}\Big)^{q}.$$
\end{lemma}
\begin{proof}
	Let $\xi\in S^{(\varphi,c)}(Q)$. We can find a $\varphi$-length John curve $\gamma$ joining $\xi$ to $x_0$ in $\Omega$ so that $\gamma (t_Q)\in Q$ for some $t_Q\in [0,1]$. Since $\varphi$ is a continuous increasing function,
	$$d_{1}(\xi, Q) \leq d_{1}\left(\xi, \gamma\left(t_{Q}\right)\right)\leq
	l(\gamma[0,t_Q]))\leq
	\varphi\left(c d_{1}\left(\gamma\left(t_{Q}\right), \partial \Omega\right)\right) \leq \varphi(c(\diam Q+d_{1}(Q,\partial\Omega))) .$$
	Note that with	 $b=\frac{aC_1}{c_1\delta}$, we have
	$$d_{1}(Q,\partial\Omega)\leq aC_1\delta^{k-1}=bc_1\delta^k\leq b(\diam Q),$$
	Thus,$$d_{1}(\xi, Q) \leq \varphi(c(\diam Q+b(\diam Q))).$$
	Then,$$Q\subset B(\xi,2\varphi((c(b+1)+1)\diam Q)).$$
	Moreover,
	\begin{align*}
	\diam(S^{(\varphi,c)}(Q)) &= \sup\{d_{1}(x,y)|x,y\in S^{(\varphi,c)}(Q)\} \\
	&\leq \sup\{d_{1}(x,Q)+\diam Q+d_{1}(y,Q)|x,y\in S^{(\varphi,c)}(Q)\} \\
	&\leq 2\varphi(c(\diam Q+b(\diam Q)))+\varphi(\diam Q)\\
	&\leq 3\varphi((c(b+1)+1)\diam Q)=3\varphi(C\diam Q).
	\end{align*}
	Now, fix $\xi\in \partial\Omega$ and define
	$$a_k=\#\{Q\in \mathscr{W}^k_{\Omega}:\xi\in S^{(\varphi,c)}(Q)\}=\#\{Q\in \mathscr{W}^k_{\Omega}:Q\in P^{(\varphi,c)}(\xi)\},k\in \mathbb{Z}. $$
	Since the cubes $Q\in \mathscr{W}_{\Omega}^k$ are essentially disjoint, we have
	\begin{align*}
	a_k\cdot \frac{1}{C_{q}}(c_1\delta^k)^q &\leq a_k\cdot \nu_{1}(B(x,c_1\delta^k))\leq\sum_{Q\in \mathscr{W}_{\Omega}^k\bigcap P^{(\varphi,c)}(\xi)}\nu_{1}(Q)\\
	&\leq\nu_{1}(B(\xi,2\varphi((c(b+1)+1)\diam Q))\cap\Omega)\\
	&\leq C_{q}\Big(2\varphi((c(b+1)+1)\diam Q)\Big)^q \\
	&\leq C_{q}\Big(2\varphi(2(c(b+1)+1)C_1\delta^k)\Big)^q.
	\end{align*}
	Then,
	$$a_k\leq C_{q}^{2}\Big(\frac{2\varphi(2CC_1\delta^k)}{c_1\delta^k}\Big)^{q}.$$
\end{proof}
\begin{lemma}\label{cedu}
	Assume that $\mu$ is a Borel measure on $\partial \Omega$ and $E\subset \partial\Omega$ is measurable. Then for each $k\in\mathbb{Z}$ we have
	$$\sum_{Q\in \mathscr{W}^k_{\Omega}}\mu(S^{(\varphi,c)}_E(Q))\leq C_{q}^{2}\Big(\frac{2\varphi(CC_1\delta^k)}{c_1\delta^k}\Big)^{q}\mu(E).$$
\end{lemma}
\begin{proof}
	Lemma \ref{geshu} implies that
	\begin{align*}
	\sum_{Q\in \mathscr{W}^k_{\Omega}}\mu(S^{(\varphi,c)}_E(Q)) &=\sum_{Q\in \mathscr{W}^k_{\Omega}}\int_E\chi_{S^{(\varphi,c)}_E(Q)}(\omega)d\mu(\omega)\\
	&=\int_E\sum\limits_{Q\in \mathscr{W}^k_{\Omega}}\chi_{S^{(\varphi,c)}_E(Q)}(\omega)d\mu(\omega)  \leq C_{q}^{2}\Big(\frac{2\varphi(C(c)C_1\delta^k)}{c_1\delta^k}\Big)^{q}\mu(E).
	\end{align*}
\end{proof}
Now we give the definition of discrete length of a curve.
\begin{definition}\label{disctete length}
	Let $\Omega\subset X_{1}$ be a $\varphi$-length John domain with center $x_{0}$. Assume that $\xi\in\partial\Omega$ and $\gamma\in I^{(\varphi,c)}(\xi,x_{0}).$ Given a continuous mapping $f:\Omega\rightarrow X_{2}$, we define the discrete length of $f(\gamma)$ by $$\ell_{d}[f(\gamma)]:=\sum_{Q\in \mathscr{W}_{\Omega},Q\cap\gamma\neq\emptyset}\diam f(Q).$$
\end{definition}
\begin{lemma}
	If $\ell_{d}[f(\gamma)]<\infty$, then $\lim\limits_{t\rightarrow 0^{+}}f(\gamma(t))$ exists.
\end{lemma}
\begin{proof}
Fix $\epsilon >0$. As
\[
\ell_{d}[f(\gamma)]=\sum_{\substack{Q\in \mathscr{W}_{\Omega}\\Q\cap\gamma\neq\emptyset}}\diam f(Q)<\infty,
\]
we deduce that \[\sum_{k=1}^{\infty}\sum_{\substack{Q\in \mathscr{W}_{\Omega}^k\\Q\cap\gamma\neq\emptyset}}\diam f(Q)<\infty.\]
Thus, there exists a constant $N=N(\epsilon)$ such that \[\sum_{k=N}^{\infty}\sum_{\substack{Q\in \mathscr{W}_{\Omega}^k\\Q\cap\gamma\neq\emptyset}}\diam f(Q)<\epsilon.\]

We denote by $\gamma_0$ the part of $\gamma$ that intersects the dyadic cubes $Q\in \mathscr{W}_{\Omega}^{k}$ with $k\leq N$. Then, by Lemma \ref{geshu}, for fixed $k$, $\#\{Q\in \mathscr{W}_{\Omega}^k:\xi \in S^{(\varphi,c)}(Q)\}\leq C(k,q)$. Then, the number of  dyadic cubes meets $\gamma_0$ is finite. Thus the distance between $\gamma_0$ and $\partial \Omega$ is positive.

Choose any sequence of positive numbers $\{t_i\}_{i=1}^{\infty}$ tending to zero. Then, there exists $N^{\prime}$ such that for any $i\geq N^{\prime}$, $\gamma(t_i)$ does not belong to $\gamma_0$. So these points belong to $\mathscr{W}_{\Omega}^{k}$ with $k\geq N+1$. Fix any $m,l\geq N^{\prime}$. As the part of $\gamma$ between $\gamma(t_m)$ and $\gamma(t_l)$ only meets finitely many Dyadic-Whitney cubes belonging to $\mathscr{W}_{\Omega}^{k}$ with $k\geq N+1$,\[d(f(t_m),f(t_l))\leq \sum_{k=N}^{\infty}\sum_{\substack{Q\in \mathscr{W}_{\Omega}^k\\Q\cap\gamma\neq\emptyset}}\diam f(Q)<\epsilon.\]
Therefore, $\{f(\gamma (t_i))\}$ is a Cauchy sequence and hence the limit exists.

\end{proof}
Next, we prove our first main result.
\begin{proof}[Proof of Theorem \ref{thm:limit exist}]
	Fix any $1<\sigma\leq \lambda_0$, where $\lambda_0$ is given by Lemma \ref{lambda}. For simplicity, we write $S_{E}(Q)$ for $S^{(\varphi,c)}_{E}(Q).$ Our aim is to prove that $\mathcal{H}^{h}(A_{\infty})=0,$ where $A_{\infty}$ is the set of points $\xi\in\partial\Omega$ for which there is a curve $\gamma\in I^{(\varphi,c)}(\xi,x_{0})$ such that $\ell_{d}[f(\gamma)]=\infty$. Since $E_{f}\subset A_{\infty}$, $\mathcal{H}^{h}(E_{f})=0.$
	
On the contrary, we assume that $\mathcal{H}^{h}(A_{\infty})>0$. Then $\mathcal{H}^{h}(A_{k})>0$, where $A_{k}$ is the set of points  $\xi\in\partial\Omega$ for which there exists $\gamma\in  I^{(\varphi,c)}(\xi,x_{0})$ so that $\ell_{d}[f(\gamma)]\geq k$. Then by Frostman's lemma (see for instance \cite[Theorem 2]{bo05}), there exists a Borel measure $\mu$ supported in $A_{k}$ so that for every $B(x,r)\subset X_{1},$ \begin{equation}\label{160}
\mu(B(x,r))\leq h(r),\end{equation}
and $$\mu(A_{k})\approx\mathcal{H}^{h}_{\infty}(A_{k})\geq\mathcal{H}^{h}_{\infty}(A_{\infty})>0,$$
where $\mathcal{H}^{h}_{\infty}(A_k)$ is the usual Hausdorff $h$-content of $A_k$. By the definition of $A_{k}$ and the definition of the discrete length,
\begin{align*}
\mu(A_{k})k\leq \int_{A_{k}}\ell_{d}[f(\gamma_{\omega})]d\mu(\omega)\leq\int_{A_{k}}\sum_{Q\in \mathscr{W}_{\Omega},Q\cap\gamma_{\omega}\neq\emptyset}\diam f(Q) d\mu(\omega),
	\end{align*}
where $\gamma_{\omega}$ is the curve associated with $\omega$ in the definition of $A_k$.
From now on, we assume that $f\in\mathcal{A}_{2}(\sigma)$. Then we have the following inequalities:
	\[
	\begin{aligned}
	\mu(A_{k})k & \leq\int_{A_{k}}\sum_{Q\in \mathscr{W}_{\Omega},Q\cap\gamma_{\omega}\neq\emptyset}\diam f(B^{Q})d\mu(\omega)\\
	& \leq\sum_{Q\in \mathscr{W}_{\Omega}}\int_{A_{k}}\chi_{S(Q)}(\omega)d\mu(\omega)\diam f(B^{Q})\leq\sum_{Q\in \mathscr{W}_{\Omega}}\mu(S_{A_{k}}(Q))\diam f(B^{Q}). \\
	\end{aligned}
    \]
	Since $f\in\mathcal{A}_{2}(\sigma),$ for every $B=B(x,r)$ for which $\sigma B\subset\subset\Omega,$
	$$\diam f(B)\leq C\bigg(\int_{\sigma B}\alpha(x)d\nu_{1}(x)\bigg)^{\frac{1}{q}}\log\bigg(\frac{1}{\diam B}\bigg)^{\frac{1}{q}}.$$
By \eqref{150}, we have that $\sigma B^Q\subset\subset \Omega$.
Then,
	\begin{align*}
	\mu(A_{k})k & \leq C\sum_{Q\in \mathscr{W}_{\Omega}}\bigg(\mu(S_{A_{k}}(Q))\log^{\frac{1}{q}}\bigg(\frac{1}{\diam B^{Q}}\bigg)\bigg)\times \Big(\int_{\sigma B^{Q}}\alpha(x)d\nu_{1}(x)\Big)^{\frac{1}{q}} \\
	& \leq C\Big(\sum_{Q\in \mathscr{W}_{\Omega}}\mu(S_{A_{k}}(Q))^{\frac{q}{q-1}}\log^{\frac{1}{q-1}}\bigg(\frac{1}{\diam B^{Q}}\bigg)\Big)^{\frac{(q-1)}{q}}\times\Big(\sum_{Q\in \mathscr{W}_{\Omega}}\int_{\sigma B^{Q}}\alpha(x)d\nu_{1}(x)\Big)^{\frac{1}{q}}.
	\end{align*}
	To simplify the inequalities, we use the fact that $c_1\delta ^j\leq \diam \,B^{Q}\leq 2C_1\delta ^j$ for $Q\in \mathscr{W}_{\Omega}^{j}$ and the property that $\sigma B^{Q}$ have uniformly bounded overlap. By Lemma $\ref{geshu}$, $\sigma  B^{Q}\subset\subset\Omega$ and $\alpha\in L^{1}(\Omega)$
	 we have,
	 $$\big(\sum_{Q\in
	 	\mathscr{W}_{\Omega}}\int_{\lambda B'_{Q}}\alpha(x)d\nu_{1}(x)\big)^{\frac{1}{q}}\leq C_2.$$
	Therefore,
	
	$$\mu(A_{k})k \leq C^{\prime}\Big(\sum_{j=1}^{\infty}\sum_{Q\in \mathscr{W}_{\Omega}^{j}}\mu(S_{A_{k}}(Q))^{\frac{q}{q-1}}j^{\frac{1}{q-1}}\Big)^{\frac{(q-1)}{q}},$$ where $C^{\prime}=C\big(C_2,   C_{1},c_{1}\big)$.
	Using Lemma \ref{cedu},
	$$\sum_{Q\in \mathscr{W}^j_{\Omega}}\mu(S_{A_{k}}(Q))\leq C_{q}^{2}\Big(\frac{2\varphi(CC_1\delta^j)}{c_1\delta^{j}}\Big)^{q}\mu(A_{k}).$$
Then, by Lemma~\ref{geshu}, \[\diam (S(Q))\leq 3\varphi(C\diam Q)\leq 3\varphi(2CC_1\delta^j).\] Thus, by \eqref{160}, we have the following estimates:
	
	\begin{align*}
	\sum_{Q\in \mathscr{W}_{\Omega}^{j}}\mu(S_{A_{k}}(Q))^{\frac{q}{q-1}}
	& \leq
	\max_{Q\in \mathscr{W}_{\Omega}^{j}}\mu(S_{A_{k}}(Q))^{\frac{1}{q-1}}\sum_{Q\in \mathscr{W}_{\Omega}^{j}}\mu(S_{A_{k}}(Q))\\
	& \leq
	C_{q}^{2}\Big(\frac{2\varphi(2CC_1\delta^j)}{c_1\delta^j}\Big)^{q}\max_{Q\in \mathscr{W}_{\Omega}^{j}}\mu(S(Q))^{\frac{1}{q-1}}\mu(A_{k}) \\
	& \leq C_{q}^{2}\Big(\frac{2\varphi(2CC_1\delta^j)}{c_1\delta^j}\Big)^{q}h\big(3\varphi (2CC_1\delta ^j)\big)^{\frac{1}{q-1}}\mu(A_{k}).
	\end{align*}
	Putting these estimates together gives $$\mu(A_{k})k\leq C\mu(A_{k})^{\frac{(q-1)}{q}}\Big(\sum_{j=1}^{\infty}C_{q}^{2}[\frac{2\varphi(2CC_1\delta^j)}{c_1\delta^j}]^{q}h\big(3\varphi (2CC_1\delta ^j)\big)^{\frac{1}{q-1}}j^{\frac{1}{q-1}}\Big)^{\frac{(q-1)}{q}}.$$
By \eqref{210} and the doubling property of $h$, we know that the second term on the right side is finite and independent of $k$. Then, by $\mu(A_k)<\infty$, we get $\mu^{\frac{1}{q}}(A_{k})k\leq C^{\prime}$ for every $k\in\mathbb{N}$, where $C^{\prime}$ is independent of $k$. Therefore, $\mu(A_{k})\rightarrow 0$ as $k\rightarrow \infty$. This is a contradiction. For $f\in\mathcal{A}_{1}(\sigma)$ we do not get the term $j^{\frac{1}{(q-1)}}$ in the last inequality and by an obvious modification of the above proof we can prove the similar conclusion.
\end{proof}
\subsection{Uniqueness of limits along John curves}\label{Uniqueness of limits along John curves}
From Section $\ref{Existence of limits along curves}$, if $\Omega$ is a $\varphi$-length John domain, then we may define an extension $\tilde{f}:\partial\Omega\rightarrow X_{2}$  of $f$ that satisfies $$\tilde{f}(\omega)=\lim_{t\rightarrow 0^{+}}f(\gamma_{\omega}(t)),$$ where $\gamma_{\omega}$ is a $\varphi$-length John curve connecting $x_{0}$ and $\omega$. As observed in \cite[Example 3.7]{ag17}, this extension might not be well-defined as the limit might be different if we change the $\varphi$-length John curve that connects $x_0$ and $\omega$.

In the rest of this section, we write $I^{c}(\xi,x_{0})$ to denote the class of all $c$-John curves connecting $\xi$ to $x_{0}$ and $$S^{c}(Q)=\{\xi\in \partial\Omega:Q\in P^{c}(\xi)\} $$ where $P^{c}(\xi)=\{Q\in \mathscr{W}_{\Omega}:Q\cap\gamma\neq\emptyset \text{ for some } \gamma\in I^{c}(\xi,x_{0})\}.$ We will show that if $\Omega$ is a $c$-uniform domain with center $x_{0},$ then $f$ can be extended to $\partial \Omega$ in a unique way along $c$-John curves.

Next, we need the growth of quasihyperbolic distance in the metric space setting. Since the proof is completely similar to the Euclidean case (see \cite[Lemma 5.2]{g15}), we omit it here.
\begin{lemma}\label{komega}
	Let $\Omega \subset X_{1}$ be a uniform domain. Then there exists a positive constant $C$, depending only on the data, such that
	$$
	k_{\Omega}(x, y) \leq C \int_{\min \{d_{1}(x, \partial \Omega), d_{1}(y, \partial \Omega)\}}^{C d_{1}(x, y)}\frac{1}{s}d s+2
	$$
	for each pair $x, y$ of points in $\Omega$.
\end{lemma}

\begin{lemma}\label{Geshu}
	Let $\Omega \subset X_{1}$ be a $c_{0}$ uniform domain with center $x_{0}$ and $ \mathscr{W}_{\Omega}$ be a Dyadic-Whitney decomposition of $\Omega .$ Then there exists a constant $\mathrm{C}$ such that for any $\mathrm{s}>0$ and pair $x_{1}, x_{2} \in \Omega$ with $d_{1}(x_{2}, \partial \Omega) \geq s, d_{1}(x_{1}, \partial \Omega) \geq s$ and $d_{1}(x_{1},x_{2}) \leq c_{1}s$ there exists a chain of Dyadic-Whitney cubes
	$\left\{Q_{k}\right\}_{k=1}^{N}$ connecting points $x_{1}$ and $x_{2}$ such that the number of cubes is uniformly bounded with respect to $s$, i.e. $N \leq C$, where $C$ depends only on $q,c_{0},c_{1}$ but not on s.
\end{lemma}
\begin{proof}
	Fix $x_{1},x_{2}\in\Omega$ satisfying the conditions in the lemma. Then we can connect $x_{1}, x_{2}$ by a quasihyperbolic geodesic $[x_{1},x_{2}]_{k}$ and get a chain of Dyadic-Whitney cubes $Q \in  \mathscr{W}_{\Omega}$ that intersect $[x_{1},x_{2}]_{k}$ in a such way that $x_{1}\in Q_{1},x_{2}\in Q_{N}$ and $Q_{i}\cap Q_{i+1}\neq\emptyset$ for every $i$. By \cite{bhk01} we know that the number of cubes in this chain is comparable to $k_{\Omega}(x_{1}, x_{2})$, joining $x_1$ to $x_2$. Since $\Omega$ is uniform, by Lemma \ref{komega}, there exists a constant $C>0$ which does not depend on $s$ such that
	$$
k_{\Omega}(x_1, x_2) \leq C \int_{\min \{d_{1}(x_1, \partial \Omega), d_{1}(x_2, \partial \Omega)\}}^{C d_{1}(x_1, x_2)}\frac{1}{s}d s+2.
	$$
Thus, the claim follows.
\end{proof}

\begin{proof}[Proof of Theorem \ref{thm:limit is unique}]
	Fix any $1<\sigma\leq \lambda_0$, where $\lambda_0$ is given by Lemma \ref{lambda}. According to the proof of Theorem~\ref{thm:limit exist}, we know that for $\mathcal{H}^{h}$-almost every point $\xi\in\partial\Omega$, $f$ has a limit for every $\gamma\in I^{c}(\xi,x_{0})$. Fix $\xi\in \partial\Omega$ as above. Then, we show that along two such curves $\gamma,\eta$, the limits of $f$ are the same.
Given a constant $s>0$, let $t_{1}, t_{2} \in[0,1]$ be such that \begin{equation}\label{230}\ell(\gamma |_{(\gamma\left(t_{1}\right),w)})=s \quad\text{ and }\quad
 \ell(\gamma |_{(\eta\left(t_{1}\right),w)})=s.\end{equation} Since $\gamma, \eta$ are both $c$-John curves, this implies that
\begin{equation}\label{220}
	cd_{1}\left(\gamma\left(t_{1}\right), \partial \Omega\right) \geq {s} \quad\text { and }\quad  cd_{1}\left(\eta\left(t_{2}\right), \partial \Omega\right) \geq {s} .
\end{equation}
	 Then we can find a chain of Dyadic-Whitney cubes $\left\{Q_{i}\right\}$, connecting $\gamma(t_1)$ and $\eta(t_2)$, and the number of the cubes is uniformly bounded by Lemma $\ref{Geshu}$. Set $y_{1}=\gamma\left(t_{1}\right)$ and $y_{N+1}=\eta\left(t_{2}\right)$. Choose $y_{i} \in Q_{i}$($2\leq i\leq N$) so that $y_i$,$y_{i+1}$ belong to the same ball $B^{Q_i}$ for all $1\leq i\leq N$.
	Since $f\in \mathcal{A}_{2}(\sigma)$,
	\begin{align*}
	d_{2}(f\left(\gamma\left(t_{1}\right)\right),f\left(\eta\left(t_{2}\right)\right))&\leq \sum_{i=1}^{N}d_{2}(f\left(y_{i+1}\right),f\left(y_{i}\right))\\
	& \leq C_{1} \sum_{i=1}^{N}\left(\int_{\sigma B^{Q_{i}}} \alpha(y) \mathrm{d}\nu_{1}(y)\right)^{1 / q} \log \bigg( \frac{1}{\diam Q_{i}}\bigg)^{1 / q} \\
	& \leq c\left(C, C_{1}\right) \max _{1 \leq i \leq N}\left(\int_{{\sigma} B^{Q_{i}}} \alpha(y) \mathrm{d}\nu_{1}(y)\right)^{1 / q} \log \bigg( \frac{1}{\diam Q_{i}}\bigg)^{1 / q}.
	\end{align*}
	
We continue the proof by contradiction. Assume that these two limits are not the same. Then, there exists $\delta>0$ so that the two limits are at distance $\geq 4\delta$. By choosing $s$ small enough, we can assume that \[d_{2}(f\left(\gamma\left(t_{1}\right)\right),f\left(\eta\left(t_{2}\right)\right))\geq \delta.\] Thus, there exists a cube $Q_s$ associated with $s$ such that \begin{equation}\label{240}\int_{{\sigma} B^{Q_s}} \alpha(y) \mathrm{d}\nu_{1}(y) \log \bigg( \frac{1}{\diam Q_s}\bigg)\geq C(\delta,q).\end{equation}
As the number of cubes is less than $N$, the diameter of this cube are all comparable to $d(\gamma (t_1),\partial \Omega)$. Moreover, \eqref{220} and \eqref{230} tell us that
$d(\gamma (t_1),\partial \Omega)\approx s$.
Therefore, we can choose a sequence $\{s_i\}$ such that the cube $Q_i$ associated with $s_i$ satisfies $\diam Q_i\approx e^{-i}$.
Thus, by \eqref{240}, \[ \frac{1}{i}\lesssim\int_{{\sigma} B^{Q_i}} \alpha(y) \mathrm{d}\nu_{1}(y).\]

\textbf{Claim}: The number of Dyadic-Whitney cubes that are the same in the sequence is less than a constant.

Indeed, assume $Q_i,\cdots,Q_{i+k}$ are the same Dyadic-Whitney cubes. It suffices to prove that $k\leq C$. By the definition of Dyadic-Whitney cubes, we know
$$e^{-i}\leq C\diam Q_i=C\diam Q_{i+k}\leq C^2e^{-i-k}.$$
Thus, $e^k\leq C^2$, concluding the proof of claim.

By the choose of $\sigma$, Lemma~\ref{lambda} and the above claim, we infer that
\[\infty=\sum_i\int_{\sigma B^{Q_i}}\alpha(y) \mathrm{d}\nu_{1}(y)\leq C\int_\Omega \alpha(y) \mathrm{d}\nu_{1}(y)<\infty,\]
which is a contradiction. The proof is thus complete.

\end{proof}

\medskip
\textbf{Acknowledgments}. The authors would like to thank Prof. Chang-Yu Guo for posing this question and for many useful conservations. They also thank Lin Cao, Yu-Heng Lan and Li-Jie Fang for many useful conservations.


\begin{thebibliography}{99}
	\bibitem{a14}
	T. \"Akkinen, \textit{Radial limits of mappings of bounded and finite distortion}. J. Geom. Anal. 24 (2014), no. 3, 1298-1322.
	
	\bibitem{ag17}
	T. \"Akkinen and C.-Y. Guo,
	\textit{Mappings of finite distortion: boundary extensions in uniform domains}.
	Ann. Mat. Pura Appl. (4) 196 (2017), no.~1, 65-83.
	

     \bibitem{bkr07}
    Z. Balogh, P. Koskela and S. Rogovin, \textit{Absolute continuity of quasiconformal mappings on curves.} Geom. Funct. Anal. 17 (2007), no. 3, 645-664.

      \bibitem{bo05}
    J. Bj\"orn and J. Onninen, \textit{Orlicz capacities and Hausdorff measures on metric spaces}. Math. Z. 251 (2005), no. 1, 131-146.
	
	 \bibitem{bhk01}
	M. Bonk, J. Heinonen and P. Koskela,
	\textit{Uniformizing gromov hyperbolic spaces}.
 	Asteisque No. 270 (2001).
 	
 	\bibitem{c06}
 	M. Cristea, \textit{Quasiregularity in metric spaces}. Rev. Roumaine Math. Pures Appl. 51 (2006), no. 3, 291-310.
 	
 	\bibitem{c23}
 	M. Cristea, \textit{On the radial limits of mappings on Riemannian manifolds}. Anal. Math. Phys. 13 (2023), no. 4, Paper No. 60.
	
	\bibitem{g15}
      C.-Y. Guo,
    \textit{Uniform continuity of quasiconformal mappings onto generalized John domains}. Ann. Acad. Sci. Fenn. Math. 40 (2015), no. 1, 183-202.

    \bibitem{g15b}
    C.-Y. Guo, \textit{Mappings of finite distortion between metric measure spaces}. Conform. Geom. Dyn. 19 (2015), 95-121.

    \bibitem{gk14}
    C.-Y. Guo and P. Koskela, \textit{Generalized John disks.} Cent. Eur. J. Math. 12 (2014), no. 2, 349-361.

    \bibitem{gw16}
    C.-Y. Guo and M. Williams, \textit{The branch set of a quasiregular mapping between metric manifolds.} C. R. Math. Acad. Sci. Paris 354 (2016), no. 2, 155-159.

    \bibitem{gx23}
    C.-Y. Guo and Y. Xuan, \textit{A note to ``Radial limits of quasiregular local homeomorphisms"},  Pure Appl. Funct. Anal., to appear 2023.

    \bibitem{hak00}
    P. Hajlasz and P. Koskela, \textit{Sobolev met Poincar\'e}. Mem. Amer. Math. Soc. 145 (2000), no. 688.

    \bibitem{hh97}
    J. Heinonen and I. Holopainen, \textit{Quasiregular maps on Carnot groups}. J. Geom. Anal. 7 (1997), no. 1, 109-148.


    \bibitem{hk98}
    J. Heinonen and P. Koskela, \textit{Quasiconformal maps in metric spaces with controlled geometry}. Acta Math. 181 (1998), no. 1, 1-61.

    \bibitem{hkst15}
    J. Heinonen, P. Koskela, N. Shanmugalingam and J. Tyson, \emph{Sobolev spaces on metric measure spaces. An approach based on upper gradients.} New Mathematical Monographs, 27. Cambridge University Press, Cambridge, 2015.

    \bibitem{hr02}
    J. Heinonen and S. Rickman, \textit{Geometric branched covers between generalized manifolds}. Duke Math. J. 113 (2002), no. 3, 465-529.

    \bibitem{hk14}
     S. Hencl and P. Koskela, \textit{Lectures on mappings of finite distortion}. Lecture Notes in Mathematics, 2096. Springer, Cham, 2014.

    \bibitem{hk12}
    T. Hyt\"onen and A. Kairema,
    \textit{Systems of dyadic cubes in a doubling metric space}.
    Colloq. Math. 126 (2012), no. 1, 1-33.

    \bibitem{im01}
    T. Iwaniec and G. Martin, \textit{Geometric function theory and non-linear analysis}. Oxford Mathematical Monographs. The Clarendon Press, Oxford University Press, New York, 2001.

    \bibitem{iko01}
    T. Iwaniec, P. Koskela and J. Onninen, \textit{Mappings of finite distortion: monotonicity and continuity}. Invent. Math. 144 (2001), no. 3, 507-531.


    \bibitem{kn08}
    P. Koskela and T. Nieminen, \textit{Homeomorphisms of finite distortion: discrete length of radial images}. Math. Proc. Cambridge Philos. Soc. 144 (2008), no. 1, 197-205.

    \bibitem{mr72}
    O. Martio and S. Rickman, \textit{Boundary behavior of quasiregular mappings}. Ann. Acad. Sci. Fenn. Ser. A. I. 1972, no. 507.

    \bibitem{ms75}
    O. Martio and U. Srebro, \textit{Automorphic quasimeromorphic mappings in $\R^n$}. Acta Math. 135 (1975), no. 3-4, 221-247.

    \bibitem{or09}
    J. Onninen and K. Rajala, \textit{Quasiregular mappings to generalized manifolds}. J. Anal. Math. 109 (2009), 33-79.


    \bibitem{r08}
    K. Rajala, \textit{Radial limits of quasiregular local homeomorphisms}. Amer. J. Math. 130 (2008), no. 1, 269-289.

    \bibitem{r93}
    S. Rickman, \textit{Quasiregular mappings}. Ergebnisse der Mathematik und ihrer Grenzgebiete (3) [Results in Mathematics and Related Areas (3)], 26. Springer-Verlag, Berlin, 1993.

    \bibitem{r89}
    Yu.G. Reshetnyak, \textit{Space mappings with bounded distortion.} Translated from the Russian by H. H. McFaden. Translations of Mathematical Monographs, 73. American Mathematical Society, Providence, RI, 1989.


     \bibitem{w12}
     M. Williams, \textit{Geometric and analytic quasiconformality in metric measure spaces}. Proc. Amer. Math. Soc. 140 (2012), no. 4, 1251-1266.

     \bibitem{w14}
     M. Williams, \textit{Dilatation, pointwise Lipschitz constants, and condition $N$ on curves}. Michigan Math. J. 63 (2014), no. 4, 687-700.



\end{thebibliography}
\end{document}